\documentclass[11pt]{article}

\title{Local regularity criteria in terms of one velocity component for  the Navier-Stokes equations}
\author{Kyungkeun Kang\thanks{
		Department of Mathematics, Yonsei University, 03722 Seoul, Republic of Korea. E-mail address: \url{kkang@yonsei.ac.kr}}
	\and
	Dinh Duong Nguyen\thanks{
		Department of Mathematics, Yonsei University, 03722 Seoul, Republic of Korea. E-mail address:  \url{nguyendinhduong.math.khtn@gmail.com}
	}
}

\usepackage{mathtools}
\usepackage{amssymb,bbm,amsmath,amsfonts,amsthm}
\usepackage{vmargin} 
\usepackage[linktocpage,colorlinks=true,raiselinks]{hyperref}
\hypersetup{urlcolor=black, citecolor=red, linkcolor=blue}
\usepackage{cleveref}
\usepackage{tkz-euclide}

\numberwithin{equation}{section}

\DeclareMathOperator*{\esssup}{ess\,sup}

\newtheorem{theorem}{Theorem}[section]

\newtheorem{lemma}{Lemma}[section]

\newtheorem{definition}{Definition}[section]
\theoremstyle{definition}
\newtheorem{remark}{Remark}[section]
\newenvironment{AMS}{}{}
\newenvironment{keywords}{}{}

\begin{document}
	\maketitle
	
	\begin{abstract}
		This paper is devoted to presenting new interior regularity criteria in terms of one velocity component for weak solutions to the Navier-Stokes equations in three dimensions. It is shown that the velocity is regular near a point $z$ if its scaled $L^p_tL^q_x$-norm of some quantities related to the velocity field is finite and the scaled $L^p_tL^q_x$-norm of one velocity component is sufficiently small near $z$.
	\end{abstract}
	
	
	
	\allowdisplaybreaks 
	
	%
	\section{Introduction}
	%
	
	Let us consider the partial regularity of weak solutions for the three dimensional incompressible Navier-Stokes equations
	
	\begin{equation*} \label{NS} \tag{NS}
		\partial_t v - \Delta v + v \cdot \nabla v + \nabla \pi = 0,
		\quad \text{div}\, v = 0 \qquad \text{in } \Omega \times (0,T),
	\end{equation*}
	associated with the divergence-free initial data $v_0$, where $\Omega \subseteq \mathbb{R}^3$ and $T > 0$. The unknowns in \eqref{NS} are the velocity $v = (v_1,v_2,v_3) : \Omega \times (0,T) \to \mathbb{R}^3$ and the pressure $\pi : \Omega \times (0,T) \to \mathbb{R}$. It is well known that \eqref{NS} is invariant under the following natural scaling 
	$$
		v(x,t) \mapsto v_\lambda(x,t) = \lambda v(\lambda x, \lambda^2 t) \quad \text{and} \quad \pi(x,t) \mapsto \pi_\lambda(x,t) = \lambda^2 \pi(\lambda x, \lambda^2 t) \quad \lambda > 0,
	$$
	which plays an important role in the regularity theory.
	
	The existence of global weak solutions $v \in L^\infty(0,T; L^2(\Omega)) \cap L^2(0,T; H^1(\Omega))$ of the initial boundary value problem to \eqref{NS} was provided a long time ago by Leray \cite{L1934} ($\Omega = \mathbb{R}^3$) and Hopf \cite{H1951} ($\Omega$ bounded) for finite energy initial data. These solutions  satisfy the global energy inequality and now are known as Leray-Hopf weak solutions. However, the uniqueness of global weak solutions and the existence of global strong solutions are still outsanding open problems. 
	
	Various sufficient conditions have been provided to guarantee the regularity (and uniqueness) of weak solutions. One of the first and of the most well known of such conditions was indepently given by Prodi \cite{P1959}, Serrin \cite{S1962, S1963} and Ladyzhenskaya \cite{L1967} (also known as the (LPS) regularity condition, see also \cite{FJR1972}) which reads $v \in L^p(0,T;L^q(\mathbb{R}^3))$ where
	\begin{equation*} \tag{LPS} \label{LPS}
		\frac{2}{p} + \frac{3}{q} \leq 1 \qquad \text{for} \quad 3 \leq q \leq \infty.
	\end{equation*}	
	The case $(p,q) = (\infty,3)$ has been investigated later by Escauriaza, Seregin and \v{S}ver\'{a}k \cite{ISS2003} (also for interior regularity cases).
	
	For a solution $v$ to \eqref{NS} we say that $z = (x,t) \in \mathbb{R}^3 \times (0,\infty)$ is a regular point of $v$ if there exists $r > 0$ such that $v \in L^\infty(Q_r(z))$. Otherwise, it will be called a singular point. The singular set of $v$ contains all singular points. The partial regularity theory for \eqref{NS} has been initiated by Scheffer \cite{S1976,S1977} and futher improved and simplified by Caffarelli, Kohn and Nirenberg \cite{CKN1982}, Lin \cite{L1998}, Ladyzhenskaya and Seregin \cite{LS1999} and Vasseur \cite{V2007}. They showed that the one-dimensional parabolic Hausdorff measure of the set of possible interior singular points for suitable weak solutions (see Definition \ref{def_sws}) is zero (also for one-dimensional Hausdorff measure). Note that if a weak solution satisfies \eqref{LPS} then it is suitable (see \cite{GKT2006}). Interior and near boundary regularity criteria for suitable weak solutions were also provided in terms of some norms with scaled factors by Gustafson, Kang and Tsai in \cite{GKT2006,GKT2007}.
	
	Regularity criteria based only on one velocity component started in \cite{NP1999} in which the authors provided the regularity in $D \times (t_1,t_2) \subset \Omega \times (0,T)$ if one velocity component $v_3$ is bounded in $D$. Then, Neustupa, Novotn\'{y} and Penel in \cite{NNP2002} proved the regularity under a (LPS)-type condition $v_3\in L^p(0,T;L^q(\mathbb{R}^3))$ with
	\begin{equation*}
		\frac{2}{p} + \frac{3}{q} \leq \frac{1}{2} \qquad \text{for} \quad 4 \leq p < \infty, \quad 6 < q \leq \infty,
	\end{equation*}	
	This result was generalized by Kukavica and Ziane \cite{KZ2006} to
	\begin{equation*}
		\frac{2}{p} + \frac{3}{q} \leq \frac{5}{8} \qquad \text{for} \quad \frac{16}{5} \leq p < \infty, \quad \frac{24}{5} < q \leq \infty,
	\end{equation*}	
	and by Cao and Titi \cite{CT2008} to
	\begin{equation*}
		\frac{2}{p} + \frac{3}{q} < \frac{2}{3} + \frac{2}{3q} \qquad \text{for} \quad 3 < p \leq \infty, \quad  \frac{7}{2} < q \leq \infty,
	\end{equation*}	
	and also by Pokorn\'{y} and Zhou \cite{ZP2010} to
	\begin{equation*}
		\frac{2}{p} + \frac{3}{q} \leq \frac{3}{4} + \frac{1}{2q}  \qquad \text{for} \quad \frac{8}{3} \leq p < \infty, \quad \frac{10}{3} < q \leq \infty.
	\end{equation*}	
	One of the first results on blow-up criteria via one velocity component in the scaling-invariant space $L^p(0,T;\dot{H}^{\frac{1}{2} + \frac{2}{p}}(\mathbb{R}^3))$\footnote{That means $\|v_\lambda\|_{L^p \big(0,\lambda^{-2}T;\dot{H}^{\frac{1}{2} + \frac{2}{p}}(\mathbb{R}^3) \big)} = \|v\|_{L^p \big(0,T;\dot{H}^{\frac{1}{2} + \frac{2}{p}}(\mathbb{R}^3) \big)}$ for $v_\lambda(x,t) = \lambda v(\lambda x,\lambda^2 t)$ with $\lambda > 0$.} was given by Chemin and Zhang in \cite{CZ2016} for $p \in (4,6)$, then by Chemin, Zhang and Zhang in \cite{CZZ2017} for $p \in (4,\infty)$ and recently extended to $p \in [2,\infty)$ in \cite{HLLZ2019} by Han, Lei, Li and Zhao. It is equivalent that if $v_3 \in L^p (0,T; \dot{H}^{\frac{1}{2} + \frac{2}{p}}(\mathbb{R}^3))$ then $v$ is regular in $\mathbb{R}^3 \times (0,T]$ for $p \in [2,\infty)$. Note that the embedding from the homogeneous Sobolev space $\dot{H}^{\frac{1}{2} + \frac{2}{p}}(\mathbb{R}^3)$ to $L^{\frac{3p}{p-2}}(\mathbb{R}^3)$ is continuous for $p > 2$ (see for example \cite{BCD2011}), and for $2 < p < \infty$, $\big(p,\frac{3p}{p-2}\big)$ satisfies \eqref{LPS}.
	
	Very recently, two significant results have been established. Firstly, Chae and Wolf \cite{CW2021} proved the regularity under strictly \eqref{LPS}, namely $\frac{2}{p} + \frac{3}{q} < 1$ for $3 < q \leq \infty$. Meanwhile, Wang, Wu and Zhang \cite{WWZ2020} obtained \eqref{LPS} in the sense of Lorentz spaces $v_3 \in L^{p,1}(0,T;L^q(\mathbb{R}^3))$ with $2 < p < \infty$ and $3 < q < \infty$, where $L^{p,1}$ denotes Lorentz spaces with respect to time variable (note that $L^{p,1} \subsetneq L^p$ for all $p > 1$). On the other hand, in \cite{BK2019} Bae and the first author proved regularity of Type I solution satisfying \eqref{LPS} condition for one component of the velocity.
	
	Further disscusions on this direction can be found in \cite{N2018}, where the author posed a question: \textit{Whether $v_3 \in L^p(0,T;L^q(\mathbb{R}^3)$ with $(p,q)$ satisfying \eqref{LPS} is sufficiently for the regularity of solution $v$ in $\mathbb{R}^3 \times (t_1,t_2)$?} In general, it is still an open problem at the time writing this paper.
	
	The aim of this paper is to provide local regularity criteria in terms of one velocity component both for local energy and suitable weak solutions. We provide initial regularity and local criteria in Theorems \ref{theo1}, \ref{theo2} and \ref{theo3}, respectively. As a consequence, we provide a partially (the $L^p_tL^q_x$-norm of $v_3$ is needed to be small enough) positive answer to the above question for $q \in [3,\infty]$. For $z = (x,t) \in \mathbb{R}^3 \times (0,\infty)$ and $r > 0$ we denote 
	$$
		Q_r(z) := B_r(x) \times (t-r^2,t) \quad \text{and} \quad v_r(t) := \frac{1}{|B_r(x)|} \int_{B_r(x)} v(y,t) \,dy.
	$$
	 Before we state our main results, we recall a regularity criterion (see 
	 \cite[Theorem 5.1]{W2015}), which is given as follows:
	There exists an absolute constant $\epsilon_W > 0$ such that if $(v,p)$ be a suitable weak solution in $Q_r(x_0,t_0)$ satisfying
	\begin{equation}\label{Wolf-2015}
		r^{-2}\|v\|^3_{L^3(Q_r(x_0,t_0))} \leq \epsilon_W
	\end{equation}
	for some $r \in (0,1)$ then $v \in L^\infty(Q_{\frac{r}{2}}(x_0,t_0))$.

	Now we are ready to state main theorems. Our first result reads as follows.
	
	\begin{theorem} \label{theo1}
		Let $(v,\pi)$ be a local energy solution \textnormal{(}resp., suitable weak solution\textnormal{)} to \eqref{NS} in $\mathbb{R}^3 \times (0,T)$ associated with the divergence-free initial data $v_0 \in L^2_{\textnormal{uloc}}(\mathbb{R}^3)$ \textnormal{(}resp., in the sense that  $\lim_{t \to 0} \|v(t)-v_0\|_{L^2(\mathbb{R}^3)} = 0$\textnormal{)}. Let $x_0 \in \mathbb{R}^3$, $M > 0$ and $\epsilon_0 \in (0,\epsilon_W]$, where $\epsilon_W$ is the constant in \eqref{Wolf-2015}. There exists $\delta_0(M,\epsilon_0) \in (0,1)$ with the following property. For some $0 < r_0 = r_0(M,\epsilon_0) \leq \sqrt{T}$ if
		\begin{equation} \label{v0}
			\sup_{r \in (0,r_0]} r^{-1}  \|v_0\|^2_{L^2(B_r(x_0))} \leq M
		\end{equation}
		and 
		\begin{equation} \label{v3v0}
			\limsup_{r \to 0^+} r^{1-\frac{2}{p}-\frac{3}{q}} \|v_3\|_{L^p_tL^q_x(Q_r(x_0,r^2))} = 0 \qquad \text{for} \quad 1 \leq p,q \leq \infty
		\end{equation}
		hold, then 
		\begin{equation*}
			(\delta_0 r_0)^{-2} \|v\|^3_{L^3(Q_{\delta_0 r_0}(x_0,(\delta_0 r_0)^2))}  \leq \epsilon_0.
		\end{equation*}
		In particular, $v$ is regular in $\{x_0\} \times (0,(\delta_1 r_0)^2]$ for some $\delta_1(M,\epsilon_0) \in (0,1)$.
	\end{theorem}

	For regularity criteria at a given point, an additional condition on the velocity is needed, in which the result is stated as follows.
	
	\begin{theorem} \label{theo2}
		Let $(v,\pi)$ be a local energy solution \textnormal{(}resp., suitable weak solution\textnormal{)} to \eqref{NS} in $\mathbb{R}^3 \times (0,\infty)$ associated with the divergence-free initial data $v_0 \in E^2$ \textnormal{(}resp., $v_0 \in L^2(\mathbb{R}^3)$\textnormal{)}. Let $z_0 = (x_0,t_0) \in \mathbb{R}^3 \times (0,\infty)$, $M > 0$ and $\epsilon_0 \in (0,\epsilon_W]$, where $\epsilon_W$ is the constant in \eqref{Wolf-2015}. There exists $\delta_0(M,\epsilon_0) \in (0,1)$ with the following property. For some $0 < r_0 = r_0(M,\epsilon_0,p_0,q_0) \leq \sqrt{t_0}$, assume that  one of following conditions holds
		\begin{align*} 
			&(i) \sup_{r \in (0,r_0]} r^{1-\frac{2}{p_0}-\frac{3}{q_0}}  \|v-v_r\|_{L^{p_0}_tL^{q_0}_x(Q_r(z_0))} &\leq M \qquad \text{for} \quad  \frac{2}{p_0} + \frac{3}{q_0} \in [1,2), \frac{3}{2} < q_0 \leq \infty;  
			\\
			&(ii) \sup_{r \in (0,r_0]} r^{2-\frac{2}{p_0}-\frac{3}{q_0}}  \|\nabla v\|_{L^{p_0}_tL^{q_0}_x(Q_r(z_0))} &\leq M \qquad \text{for} \quad  \frac{2}{p_0} + \frac{3}{q_0} \in [2,3), 1 < q_0 \leq \infty;
			\\
			&(iii) \sup_{r \in (0,r_0]} r^{2-\frac{2}{p_0}-\frac{3}{q_0}}  \|w\|_{L^{p_0}_tL^{q_0}_x(Q_r(z_0))} &\leq M \qquad \text{for} \quad  \frac{2}{p_0} + \frac{3}{q_0} \in [2,3), 1 < q_0 < \infty;
			\\
			&(iv) \sup_{r \in (0,r_0]} r^{3-\frac{2}{p_0}-\frac{3}{q_0}}  \|\nabla w\|_{L^{p_0}_tL^{q_0}_x(Q_r(z_0))} &\leq M \qquad \text{for} \quad  \frac{2}{p_0} + \frac{3}{q_0} \in [3,4), 1 \leq q_0 \leq \infty;
		\end{align*}
		where $w = \nabla \times  v$. In addition, suppose that  
		\begin{equation} \label{v3z0}
			\limsup_{r \to 0^+} r^{1-\frac{2}{p}-\frac{3}{q}} \|v_3\|_{L^p_tL^q_x(Q_r(z_0))} = 0 \qquad \text{for} \quad 1 \leq p,q \leq \infty.
		\end{equation}
		Then 
		\begin{equation} \label{W}
			(\delta_0 r_0)^{-2} \|v\|^3_{L^3(Q_{\delta_0 r_0}(z_0))} \leq \epsilon_0.
		\end{equation}
		In particular, $z_0$ is a regular point of $v$.
	\end{theorem}
	
	In particular, due to the structure of the local energy inequality we are able to consider a one-scaled condition as in \cite[Theorems 2.1]{KWM2017,KWM2019}, which is related to the pressure for suitable weak solutions but nor for local energy solutions. More precisely, the result is stated as follows.
	
	\begin{theorem} \label{theo3}
		Let $(v,\pi)$ be a local energy solution \textnormal{(}resp., suitable weak solution\textnormal{)} to \eqref{NS} in $\mathbb{R}^3 \times (0,\infty)$ associated with the divergence-free initial data $v_0 \in E^2$ \textnormal{(}resp., $v_0 \in L^2(\mathbb{R}^3)$\textnormal{)}. Let $z_0 = (x_0,t_0) \in \mathbb{R}^3 \times (0,\infty)$, $M > 0$ and $\epsilon_0 \in (0,\epsilon_W]$, where $\epsilon_W$ is the constant in \eqref{Wolf-2015}. There exist $\epsilon(M,\epsilon_0)$ and $\delta_0(M,\epsilon_0) \in (0,1)$ with the following property. For some $0 < r_0 = r_0(M,\epsilon_0,p,q) \leq \sqrt{t_0}$ if
		\begin{align*} 
			\text{either} \quad r_0^{-2} \|v\|^3_{L^3(Q_{r_0}(z_0))}  &\leq M \qquad \text{for local energy solution}
			\\
			\text{or} \quad  
			r_0^{-2} \left(\|v\|^3_{L^3(Q_{r_0}(z_0))} + \|\pi\|^\frac{3}{2}_{L^\frac{3}{2}(Q_{r_0}(z_0))} \right) &\leq M \qquad \text{for suitable weak solution}
		\end{align*}
		and 
		\begin{equation} \label{v3r0}
			r^{1-\frac{2}{p}-\frac{3}{q}}_0 \|v_3\|_{L^p_tL^q_x(Q_{r_0}(z_0))} \leq \epsilon \qquad \text{for} \quad 1 \leq p,q \leq \infty
		\end{equation}
		hold then \eqref{W} follows and $z_0$ is a regular point of $v$. 
	\end{theorem}
	
	\begin{remark} We add some comments on our results:
		\begin{enumerate}
			\item[1.] Theorem \ref{theo1} is inspired by \cite[Theorem 3.1]{KMT2021}, where the authors provided the initial regularity for $M$ sufficiently small in \eqref{v0}. As a consequence, the same conclusion holds if \eqref{v3v0} is replaced by $v_3 \in L^p_tL^q_x(Q_{r_0}(z_0))$ for $(p,q)$, $q>3$ satisfying \eqref{LPS}. 
			
			\item[2.] Theorem \ref{theo2} extends the results in \cite[Theorem 1.1]{GKT2007} and \cite[Theorems 3.1 and 3.2]{WZ2014} in the sense that we only need one component is sufficiently small instead of the whole velocity or two components. A special case of our result shows that $v$ is regular if 
			\begin{equation*}
				|v_1,v_2| \leq \frac{M}{\sqrt{T-t}} \quad \text{and} \quad |v_3| \leq \frac{\epsilon(M)}{\sqrt{T-t}},
			\end{equation*}
			which improves Leray's result \cite[page 227]{L1934}. In \cite[Theorem 1.3]{CW2021}, the authors proved that $z_0$ is a regular point under the following assumption for some $a \in (1,\infty)$ and for some $\rho > 0$
			\begin{equation} \label{CW}
				\limsup_{r \to 0^+} \, (-\log(r))^a r^{1-\frac{2}{p}-\frac{3}{q}} \|v_3\|_{L^p_tL^q_x(Q_{\rho,r}(z_0))} = 0,
			\end{equation} 
			where $1 \leq \frac{2}{p} + \frac{3}{q} \leq \frac{3}{2}$, $3 \leq q < \infty$, $Q_{\rho,r}(z_0) = B'_{\rho}(x'_0) \times (x_{03}-r,x_{03}+r) \times (t_0-r^2,t_0)$ with $B'_{\rho}(x'_0) = \{y \in \mathbb{R}^2 : |y-x'_0| < \rho\}$ and $x_0 = (x_0',x_{03})$. It can be seen that \eqref{CW} is not given in a dimensionless form as \eqref{v3z0}, but we need one of the additional conditions on $v$. As a result, in Theorem \ref{theo2}, if one of (i),(ii),(iii) or (iv) holds for all $z_0$ and $v_3$ satisfies \eqref{LPS} with $q>3$ instead of \eqref{v3z0} then the global regularity follows.
			
			\item[3.] Theorem \ref{theo3} is inspired by \cite[Theorems 2.1]{KWM2017,KWM2019}, but we generalize the condition on $v_3$ to the scaled $L^p_tL^q_x$-norm instead of only on $L^3_{t,x}$-norm. We remark that the condition is not involved the pressure in the case of local energy solutions due to the decomposition \eqref{pressure_decomposition}. Moreover, Theorem \ref{theo3} also yields \cite[Theorem 2.2]{KWM2019}. More precisely, let $(v,\pi)$ be a Leray-Hopf weak solution to \eqref{NS} in $\mathbb{R}^3 \times (0,T)$. If $v$ is regular then for $q \geq 3$ and $t \in (0,T)$
			\begin{equation*}
				\|v_1(\cdot,t),v_2(\cdot,t)\|_{L^q(\mathbb{R}^3)} \leq \frac{M}{(T-t)^{\frac{1}{2}-\frac{3}{2q}}} \quad \text{and} \quad \|v_3(\cdot,t)\|_{L^q(\mathbb{R}^3)} \leq \frac{\epsilon(M)}{(T-t)^{\frac{1}{2}-\frac{3}{2q}}}.
			\end{equation*} 
			
			\item[4.] In \cite[Theorem 1.1]{BK2019} the authors proved that smooth bounded solutions of \eqref{NS} in $\mathbb{R}^3 \times (-1,0)$ do not blow up at $t = 0$ if
			\begin{align*}
				&\|v(\cdot,t)\|_{L^\infty(\mathbb{R}^3)} \leq \frac{C}{\sqrt{-t}} \qquad \text{for } t \in (-1,0),
				\\
				&v_3 \in L^p(-1,0;L^q(\mathbb{R}^3))  \qquad \text{for} \quad  \frac{2}{p} + \frac{3}{q} = 1, q \in (3,\infty].
			\end{align*}
			They also provided the same result in the cases $v_3$ belongs to other invariant spaces such as weak Lebesgue, homogeneous Morrey–Campanato and homogeneous Besov spaces. Note that our conditions on $v_3$ as in \eqref{v3v0}, \eqref{v3z0} or \eqref{v3r0} can be considered in the above invariant spaces. We should mention that by using \eqref{CW} the authors in \cite{CW2021} provided the global regularity in the strict case, i.e.,
			\begin{equation*}
				v_3 \in L^p(0,T;L^q(\mathbb{R}^3))  \qquad \text{for} \quad \frac{2}{p} + \frac{3}{q} < 1, q \in (3,\infty],
			\end{equation*}
			which is due to the logarithmic factor. An immediate consequence of Theorem \ref{theo3} is that if for some $\epsilon(v,\pi) > 0$
			\begin{equation*}
				\|v_3\|_{L^p(0,\infty;L^q(\mathbb{R}^3))}  \leq \epsilon \qquad \text{for} \quad \frac{2}{p} + \frac{3}{q} \leq 1, q \in [3,\infty],
			\end{equation*}
			then $v$ is a regular solution in $\mathbb{R}^3 \times (0,\infty)$. Therefore, we leave an open question whether or not the above smallness assumption or the logarithmic factor in \eqref{CW} can be relaxed.
			
			\item[5.] Theorems \ref{theo1}-\ref{theo3} partially answer the question posed in \cite{N2018} and also hold for Leray-Hopf weak solutions under some futher assumption on $v_0$, for example in the case $v_0 \in L^2(\mathbb{R}^3) \cap L^3(\mathbb{R}^3)$. Note that for our results in the case of suitable weak solutions we can replace $\mathbb{R}^3$ by bounded domains $\Omega$ in which we need to require $0 < r_0 \leq \min\{\sqrt{t_0},\text{dist}(x_0,\partial \Omega)\}$.
		\end{enumerate}
	\end{remark}

	The rest of the paper is organized as follows: Section \ref{sec:pre} is devoted to recalling notions of weak solutions to \eqref{NS}. Proofs of Theorems \ref{theo1}-\ref{theo3} will be given in Sections \ref{sec:theo1}-\ref{sec:theo3}, respectively.
	
	%
	\section{Preliminaries} \label{sec:pre}
	%
	
	We introduce some notations and recall some notions of weak solutions to (\ref{NS}). For $x \in \mathbb{R}^3$, $r, t > 0$ and $z = (x,t)$  we denote $dz = dxdt$, $Q_r(z) = B_r(x) \times (t-r^2,t)$. Let $A$ be a set and $\delta \in \mathbb{R}$ we denote $\delta A = \{\delta a : a \in A\}$. For $A = (a_{ij})$ and $B = (b_{ij})$ be two matrices we define $A:B = \sum_{i,j} a_{ij} b_{ij}$ and $u \otimes v = (c_{ij})$ where $c_{ij} = u_i v_j$ for $u = (u_1, u_2, u_3), v = (v_1, v_2, v_3)$. For convenience, we denote $\partial_i = \partial_{x_i}$ for $1 \leq i \leq 3$. We use $C$ for denoting absolute constants. We also use the usual notions for Bochner, Lebesgue and Sobolev spaces. For $q \in [1,\infty]$, we define
	$$
	L^q_{\textnormal{uloc}}(\mathbb{R}^3) := \left\{f \in L^q_{\textnormal{loc}}(\mathbb{R}^3) : \|f\|_{L^q_{\textnormal{uloc}}(\mathbb{R}^3)} := \sup_{x_0 \in \mathbb{R}^3} \|f\|_{L^q (B_1(x_0))} < \infty \right\}.
	$$
	Let $E^q$ be the closure of $C^\infty_0(\mathbb{R}^3)$ in $L^q_{\textnormal{uloc}}(\mathbb{R}^3)$-norm. In \cite{LR2002}, the space $E^q$ can be defined by an equivalent way as follows
	$$
	E^q := \left\{f \in L^q_{\textnormal{uloc}}(\mathbb{R}^3) : \lim_{|x| \to \infty} \|f\|_{L^q(B_1(x))} = 0\right\}.
	$$
	We recall the well-known Gagliardo-Nirenberg interpolation inequality (see for example in \cite{CKN1982}) for $u \in H^1(B_r(x_0))$ with $r > 0$, $x_0 \in \mathbb{R}^3$
	\begin{equation} \label{Inter_i}
		\|u\|^{q}_{L^q(B_r(x_0))} \leq C \|\nabla u\|^{2a}_{L^2(B_r(x_0))} \|u\|^{q-2a}_{L^2(B_r(x_0))} + C r^{-2a} \|u\|^q_{L^2(B_r(x_0))},
	\end{equation}
	where $2 \leq q \leq 6$ and $a = \frac{3}{4}(q - 2)$. Let us recall the notions of suitable weak solutions (see \cite{CKN1982, L1998}) and local energy solutions or local Leray solutions (see \cite{BT2020, JS2013,JS2014,KMT2020_2,KMT2021, KS2007,LR2002}) to \eqref{NS}, respectively.
	
	\begin{definition}[Suitable weak solutions] \label{def_sws}
		For any domain $\Omega \subset \mathbb{R}^3$ and open interval $I \subset (0,\infty)$, we say $(v, \pi)$ is a suitable weak solution to \eqref{NS} in $\Omega \times I$ if it satisfies
		\begin{enumerate}
			\item $v \in L^\infty(I;L^2(\Omega)) \cap L^2(I;\dot{H}^1(\Omega))$, $\pi \in L^{\frac{3}{2}}(\Omega \times I)$;
			
			\item \eqref{NS} in the sense of distributions in $\Omega \times I$;
			
			\item and the local energy inequality
			\begin{align} \label{lei}
				\int_\Omega |v|^2 \varphi \,dx + 2  \int_{\Omega_s} |\nabla v|^2 \varphi \,dz
				&\leq \int_{\Omega_s} |v|^2 (\partial_t \varphi
				+ \Delta \varphi) + 
				(|v|^2 + 2\pi) (v \cdot \nabla \varphi) \,dz
			\end{align}
			for $s \in I$, $\Omega_s = \Omega \times (0,s)$ and all non-negative $\varphi \in C^\infty_0(\Omega \times I)$.
		\end{enumerate}
	\end{definition}
	
	\begin{definition}[Local energy solutions] \label{def_les}
		A vector field $v \in L^2_{\textnormal{loc}}(\mathbb{R}^3 \times [0, T))$ is a local energy solution to \eqref{NS} with divergence free initial data $v_0 \in L^2_{\textnormal{uloc}}(\mathbb{R}^3)$ if
		\begin{enumerate}
			\item for some $\pi \in L^{\frac{3}{2}}_{\textnormal{loc}} (\mathbb{R}^3 \times [0, T))$, the pair $(v, \pi)$ is a distributional solution to \eqref{NS};
			
			\item for any $r > 0$,
			\begin{equation*} 
				\esssup_{0 \leq t < \min\{r^2, T\}} \sup_{x_0 \in \mathbb{R}^3} \int_{B_r(x_0)} |v|^2 \,dx + \sup_{x_0 \in \mathbb{R}^3} \int^{\min\{r^2,T\}}_0 \int_{B_r(x_0)}|\nabla v|^2 \,dz < \infty,
			\end{equation*}
			\item for all compact subsets $K$ of $\mathbb{R}^3$ we have $v(t) \to v_0$ in $L^2(K)$ as $t \to 0^+$;
			
			\item $(v, p)$ satisfies the local energy inequality \eqref{lei} for all non-negative functions $\varphi \in C^\infty_0(Q)$ with all cylinder $Q$ compactly supported in $\mathbb{R}^3 \times (0, T)$;
			
			\item for every $x_0 \in \mathbb{R}^3$ and $r > 0$, there exists $c_{r, x_0} \in L^\frac{3}{2}((0, T))$ such that
			\begin{align} \label{pressure_decomposition}
				\pi(x, t) - c_{r,x_0}(t) &= \frac{1}{3}|v(x, t)|^2 + \textnormal{p.v.}  \int_{B_{3r}(x_0)} K(x - y) : (v \otimes v)(y, t) \,dy
				\nonumber\\
				&\quad + \int_{\mathbb{R}^3 \setminus B_{3r}(x_0)} (K(x - y) - K(x_0 - y)) : (v \otimes v)(y, t) \,dy
			\end{align}
			in $L^\frac{3}{2}(B_{2r}(x_0) \times (0, T))$ where $K(x) := \nabla^2\big(\frac{1}{4\pi|x|}\big)$;
			
			\item and for any compactly supported functions $w \in L^2(\mathbb{R}^3)$,
			\begin{equation*}
				\text{the function}\quad t  \mapsto \int_{\mathbb{R}^3} v(x, t) \cdot w(x) \,dx \quad \text{is continuous on } [0, T).
			\end{equation*}
		\end{enumerate}
	\end{definition}
	If $(v,p$) is a local energy solution to \eqref{NS} in $\mathbb{R}^3 \times (0,T)$ for all $T \in (0,\infty)$ then we say that it is a local energy solution in $ \mathbb{R}^3 \times (0, \infty)$. We recall a version of  Gronwall inequality due to Bradshaw and Tsai.
	
	\begin{lemma}[Lemma 2.2 in \cite{BT2020}] \label{lem_BT}
		Let $f \in L^\infty_{\textnormal{loc}}([0, T_0]; [0, \infty))$ satisfies 
		$$
		f(t) \leq a + b \int^t_0 f(s) + f^m(s) \,ds \qquad \text{for } t \in [0, T_0]
		$$
		and for some constants $a,b > 0$, $m \geq 1$. Then 
		$$
		f(t) \leq 2a \quad \forall t \in [0, T] \quad \text{where} \quad  T = \min\left\{T_0, \frac{C}{b(1 + a^{m-1})} \right\}.
		$$
	\end{lemma}	
	
	
	
	%
	\section{Proof of Theorem \ref{theo1}} \label{sec:theo1}
	%
	
	\begin{proof}[Proof of Theorem \ref{theo1}]\quad $\bullet$\,\textnormal{(}The case of local energy solution\textnormal{)}\, Let $\gamma \in (0, \frac{r_0}{3}]$. The local energy inequality (see \cite[Theorem 3.4]{KMT2021}) with $\gamma \leq r \leq \frac{r_0}{3}$ 
	and $s \in (0,r^2_0]$ yields
	\begin{align*}
		E_r(s) &:= \esssup_{t \in (0,s)} \frac{1}{r} \int_{B_r(x_0)} |v(t)|^2 \,dx + \frac{1}{r} \int_0^s \int_{B_r(x_0)} |\nabla v|^2 \,dz + \frac{1}{r^2} \int_0^s \int_{B_r(x_0)} |\pi_{r_0,x_0}|^\frac{3}{2} \,dz
		\\
		&\leq 2M + \frac{C}{\gamma^2} \int_0^s \mathcal{E}_\gamma(t) \,dt + \frac{C}{r^2} \int_0^s \int_{B_{2r}(x_0)} |v|^3 \,dz + \frac{C}{r^2} \int_0^s \int_{B_{2r}(x_0)} |\pi_{r_0,x_0}|^\frac{3}{2} \,dz,
	\end{align*}
	where we used \eqref{v0} and \eqref{pressure_decomposition} with $(x,t) \in B_{2r_0}(x_0) \times (0,T)$ 
	\begin{align*}
		\pi_{r_0,x_0} &:= \pi(x,t) - c_{r_0,x_0}(t) = \frac{1}{3} |v(x,t)|^2 + p_1(x,t) + p_2(x,t),
		\\
		p_1(x,t) &:= \textnormal{p.v.} \int_{B_{3r_0}(x_0)} K(x - y) : (v \otimes v)(y, t) \,dy,
		\\
		p_2(x,t) &:= \int_{\mathbb{R}^3 \setminus B_{3r_0}(x_0)} (K(x - y) - K(x_0-y) : (v \otimes v)(y, t) \,dy,
		\\
		\mathcal{E}_\gamma(s) &:= \sup_{r \in [\gamma,r_0]} E_r(s).
	\end{align*}
	The cubic term is denoted by $R_3$ and can be estimated by using \eqref{Inter_i} and Young inequality as follows
	\begin{equation*}
		R_3 \leq \epsilon \mathcal{E}_\gamma(s) + \frac{C_\epsilon}{\gamma^2} \int^{s}_{0} \mathcal{E}^\frac{3}{2}_\gamma(t) + \mathcal{E}^3_\gamma(t) \,dt
	\end{equation*}
	for some absolute constant $\epsilon \in (0,1)$ to be specified later. For the pressure term which is denoted by $R_4$ and is bounded by using the Calder\'{o}n-Zygmund estimate
	\begin{align*}
		R_4 &\leq \frac{C}{r^2} \int_0^{s} \int_{B_{3r}(x_0)}  |v|^3 \,dz + \frac{C}{r^2} \int_0^{s}\int_{B_{2r}(x_0)} |p_2|^\frac{3}{2} \,dz.
	\end{align*}
	For any $x \in B_{2r}(x_0) \subset B_{2r_0}(x_0)$ and $y \in \mathbb{R}^3 \setminus B_{3r_0}(x_0)$ we have the fact that 
	\begin{equation*}
		|K(x-y) - K(x_0-y)| \leq C\frac{|x-x_0|}{|x_0-y|^4} \leq \frac{Cr}{|x_0 - y|^4}.
	\end{equation*}	
	Therefore, it follows that for $(x,t) \in B_{2r}(x_0) \times (0,s)$ with $s \in (0,r_0^2]$
	\begin{align*} 
		|p_2(x,t)| &\leq Cr \int_{\mathbb{R}^3 \setminus B_{3r_0}(x_0)} \frac{|v(y,t)|^2}{|x_0-y|^4} \,dy
		\\
		&\leq C \sum^\infty_{i=0} \int_{B_{2^{i+1}r_0}(x_0) \setminus B_{2^ir_0}(x_0)} \frac{|v(y,t)|^2}{|x_0-y|^4} \,dy
		\\
		&\leq C \sum^\infty_{i=0} \frac{1}{2^{4i}}\int_{B_{2^{i+1}r_0}(x_0)} |v(y,t)|^2 \,dy
		\\
		&\leq C \esssup_{t\in (0,r_0^2)} \sup_{x_0 \in \mathbb{R}^3} \int_{B_{r_0}(x_0)} |v(x,t)|^2 \,dx =: V < \infty,
	\end{align*}
	which yields
	\begin{equation*}
		R_4 \leq CV^\frac{3}{2} + \epsilon \mathcal{E}_\gamma(s) + \frac{C_\epsilon}{\gamma^2} \int^{s}_{0} \mathcal{E}^\frac{3}{2}_\gamma(t) + \mathcal{E}^3_\gamma(t) \,dt 
	\end{equation*}	
	and for $0 < \gamma < r \leq \frac{r_0}{3}$ 
	\begin{equation*}
		E_r(s) \leq 2M + CV^\frac{3}{2} + 2\epsilon \mathcal{E}_\gamma(s) + \frac{C_\epsilon}{\gamma^2} \int^{s}_{0} \mathcal{E}_\gamma(t) + \mathcal{E}^3_\gamma(t) \,dt.
	\end{equation*}	
	On the other hand, for $\frac{r_0}{3} < r \leq r_0$ 
	\begin{equation*}
		E_r(s) \leq \frac{9}{r_0} \left(\esssup_{t \in (0,r^2_0)} \int_{B_{r_0}(x_0)} |v(t)|^2 \,dx +  \int_{Q_{r_0}(x_0,r^2_0)} |\nabla v|^2 +  \frac{1}{r_0} |\pi_{r_0,x_0}|^\frac{3}{2} \,dz\right) := E_0 < \infty.
	\end{equation*}	
	Therefore, by choosing $\epsilon = \frac{1}{4}$ we obtain for all $s \in (0,r^2_0]$ and for all $\gamma \in (0,r_0]$
	\begin{equation*}
		\mathcal{E}_\gamma(s) \leq C(E_0 + M + V^\frac{3}{2}) + \frac{C}{\gamma^2} \int^{s}_{0} \mathcal{E}_\gamma(t) + \mathcal{E}^3_\gamma(t) \,dt.
	\end{equation*}	
	Applying Lemma \ref{lem_BT} to the above form one has 
	\begin{equation*}
		\mathcal{E}_\gamma(s) \leq E_1, \qquad \forall s \in (0, \lambda \gamma^2], \forall \gamma \in (0,r_0],
	\end{equation*}		
	where $E_1 := C(E_0 + M +  V^\frac{3}{2})$ and $\lambda := \frac{C}{1 + E_1^2} \leq 1$. It yields $E_r(r^2) \leq \lambda^{-1}E_1$ for all $r \in (0,\lambda r_0]$. In particular,
	\begin{equation*}
		E_r(r^2) \leq \lambda^{-1}(E_1 + E_0) +  =: E_2, \qquad \forall r \in (0,r_0].
	\end{equation*}	
	We assume that the conclusion is not true, i.e., 
	\begin{equation*}
		(\delta_0 r_0)^{-2} \|v\|^3_{L^3(Q_{\delta_0 r_0}(x_0,(\delta_0 r_0)^2))} > \epsilon_0, \qquad \forall \delta_0 \in (0,1).
	\end{equation*}	
	For $(x,t) \in Q_1(0,0)$ and $\delta \in (0,r_0]$ we define 
	\begin{equation*} 
		v_\delta(x,t) := \delta v(\delta x + x_0,\delta^2 t + \delta^2)
		\quad 
		\text{and}
		\quad
		\pi_\delta(x,t) := \delta^2 \pi_{r_0,x_0}(\delta x + x_0,\delta^2 t + \delta^2).
	\end{equation*}
	Then $(v_\delta,\pi_\delta)$ is a sequence of suitable weak solutions to \eqref{NS} in $Q_1(0,0)$ since $(v,\pi_{r_0,x_0})$ is a suitable weak solution to \eqref{NS} in $Q_\delta(x_0,\delta^2)$ for all $\delta \in (0,r_0]$. Moreover, it can be seen from the bound of $(v,\pi)$ that
	\begin{equation*}
		\|v_\delta\|^2_{L^\infty(-1,0;L^2(B_1(0)))} + \|\nabla v_\delta\|^2_{L^2(-1,0;L^2(B_1(0)))} + 	\|\pi_\delta\|^\frac{3}{2}_{L^\frac{3}{2}(-1,0;L^\frac{3}{2}(B_1(0)))} \leq E_2.
	\end{equation*}
	It implies that there exist a subsequence of $(v_\delta,\pi_\delta)$ (still denoted by $(v_\delta,\pi_\delta$)) and a pair $(\bar{v},\bar{\pi})$ and by using Aubin-Lion lemma and interpolation we find that as $\delta \to 0$
	\begin{equation*} 
		\left\{
		\begin{aligned}
			v_\delta  &\overset{\ast}{\rightharpoonup} \bar{v} \qquad \text{ in } L^\infty(-1,0;L^2(B_1(0))),
			\\
			\nabla v_\delta  &\overset{}{\rightharpoonup} \nabla \bar{v} \quad\, \text{ in } L^2(-1,0;L^2(B_1(0))),
			\\
			\pi_\delta  &\overset{}{\rightharpoonup} \bar{\pi} \qquad \text{ in } L^\frac{3}{2}(-1,0;L^\frac{3}{2}(B_1(0))),
			\\
			v_\delta &\to \bar{v} \qquad \text{ in } L^a(-1,0;L^a(B_1(0))) \qquad \text{for} \quad 1 \leq a < \frac{10}{3}.
		\end{aligned}
		\right.
	\end{equation*}
	It can be seen that $(\bar{v},\bar{\pi})$ is also a suitable weak solution to \eqref{NS} in $Q_1(0,0)$. In addition, using \eqref{v3v0} we obtain for $1 \leq p,q \leq \infty$
	\begin{equation*} \label{lim_v_delta}
		\limsup_{\delta \to 0^+} \|v_{\delta,3}\|_{L^p(-1,0;L^q(B_1(0)))} = \limsup_{\delta \to 0^+} \delta^{1-\frac{2}{p}-\frac{3}{q}} \|v_3\|_{L^p(0,\delta^2;L^q(B_\delta(x_0)))} = 0,
	\end{equation*}	 
	which implies that $\bar{v}_3 \equiv 0$. Thus, $(\bar{v},\bar{\pi})$ is also a suitable weak solution to "the limiting system" which is formally reduced from \eqref{NS} by setting $\bar{v}_3 = 0$. The authors in \cite[Theorem 3.1]{KWM2017} proved that $\bar{v}$ is regular. However, for some $0 < \nu < r_0^{-1}$ so that $0 < \delta \nu < 1$ we have
	\begin{equation*}
		(\nu r_0)^{-2} \|v_\delta\|^3_{L^3(Q_{\nu r_0}(0,0))}  = (\delta\nu r_0)^{-2} \|v\|^3_{L^3(Q_{\delta\nu r_0}(x_0,(\delta \nu r_0)^2))} > \epsilon_0,
	\end{equation*}	
	and
	\begin{equation*}
		(\nu r_0)^{-2} \|\bar{v}\|^3_{L^3(Q_{\nu r_0}(0,0))}  \geq  \epsilon_0 \quad \text{as} \quad \delta \to 0,
	\end{equation*}	
	which is a contradiction with the regularity of $\bar{v}$ by choosing $\nu$ sufficiently small. Therefore, the proof is complete by iteration and using the one scale regularity criterion given by \eqref{Wolf-2015}.
	
	$\bullet$\,\textnormal{(}The case of suitable weak solution\textnormal{)} In this case the proof is quite similar to that of for local energy solutions and we only need to estimate the pressure term with $\pi$ instaed of $\pi_{r_0,x_0}$. As in \cite[Theorem 3.1]{KMT2021}, we decompose the pressure as $\pi = p_1 + p_2$ with
	\begin{equation*}
		p_1(x,t) := - \frac{1}{3} \xi(x)|v(x,t)|^2 + \textnormal{p.v.} \int_{\mathbb{R}^3} K(x-y):(v \otimes v)(y,t) \xi(y) \,dy,
	\end{equation*}
	where $K$ is given as in \eqref{pressure_decomposition}, $\xi = \xi_{x_0}$ is a smooth cut-off function with $\xi = 1$ in $B_\rho(x_0)$ and supported in $B_{2\rho}(x_0)$ for some $\rho \in [3r,\frac{r_0}{2}]$, here $0 < \gamma \leq r \leq \frac{r_0}{8C_1}$ for some absolute constant $C_1$ will be given below and $\rho$ will be specified later. Since $p_2$ is harmonic in $B_{\rho}(x_0)$ and using the Calder\'{o}n-Zygmund estimate we find that
	\begin{align*}
		\int_{B_{2r}(x_0)} |\pi|^\frac{3}{2} \,dx 
		&\leq C\int_{B_{2r}(x_0)} |p_1|^\frac{3}{2} \,dx + C\left(\frac{r}{\rho}\right)^3 \int_{B_\rho(x_0)} |p_2|^\frac{3}{2} \,dx
		\\
		&\leq C\int_{B_{2r}(x_0)} |p_1|^\frac{3}{2} \,dx + C\left(\frac{r}{\rho}\right)^3 \int_{B_\rho(x_0)} |p_1|^\frac{3}{2} \,dx  + C\left(\frac{r}{\rho}\right)^3 \int_{B_\rho(x_0)} |\pi|^\frac{3}{2} \,dx
		\\
		&\leq C\int_{B_{2\rho}(x_0)} |v|^3 \,dx + C\left(\frac{r}{\rho}\right)^3 \int_{B_\rho(x_0)} |\pi|^\frac{3}{2} \,dx,
	\end{align*}
	which implies for some absolute constant $C_1 > 1$
	\begin{align*}
		R_4 := \frac{C}{r^2} \int^s_0 \int_{B_{2r}(x_0)} |\pi|^\frac{3}{2} \,dz &\leq \frac{C}{r^2} \int^s_0\int_{B_{2\rho}(x_0)} |v|^3 \,dz + \frac{C_1r}{\rho} \frac{1}{\rho^2} \int^s_0 \int_{B_\rho(x_0)} |\pi|^\frac{3}{2} \,dz
		\\
		&\leq \frac{C}{r^2} \int^s_0\int_{B_{8C_1r}(x_0)} |v|^3 \,dz + \frac{1}{4} \mathcal{E}_\gamma(s)
		\\
		&\leq \frac{1}{4} \mathcal{E}_\gamma(s) + \epsilon \mathcal{E}_\gamma(s) + \frac{C_\epsilon}{\gamma^2} \int^{s}_{0} \mathcal{E}^\frac{3}{2}_\gamma(t) + \mathcal{E}^3_\gamma(t) \,dt 
	\end{align*}
	by choosing $\rho = 4C_1r \in [3r,\frac{r_0}{2}]$ and using the estimate of $R_3$ with the fact that $8C_1r \in [\gamma,r_0]$. Thus the rest of the proof is almost the same that of the previous case. We omit the details and end the proof. 
	\end{proof}

	%
	\section{Proof of Theorem \ref{theo2}} \label{sec:theo2}
	%
	
	\begin{proof}[Proof of Theorem \ref{theo2}]
\quad $\bullet$\,\textnormal{(}The case of local energy solution\textnormal{)}\, 	
	The local energy inequality gives us for $0 < r \leq \frac{r_0}{3}$
	\begin{align*}
		E(r) &:= \esssup_{t \in (t_0-r^2,t_0)} \frac{1}{r} \int_{B_r(x_0)} |v(t)|^2 \,dx + \frac{1}{r} \int_{Q_r(z_0)} |\nabla v|^2 \,dz + \frac{1}{r^2} \int_{Q_r(z_0)} |\pi_{r_0,x_0}|^\frac{3}{2} \,dz
		\\
		&\leq C\left(1 +  \frac{1}{9r^2} \|v\|^3_{L^3(Q_{3r}(z_0))}  + \max\{t_0^\frac{9}{4},1\}\|v_0\|^3_{L^2_{\textnormal{uloc}}(\mathbb{R}^3)} \right) =: C(1 + A(3r) + V_0),
	\end{align*}
	where we used H\"{o}lder, Young and Calder\'{o}n-Zygmund inequalities, a similar estimate for $p_2$ as in the proof of Thereom \ref{theo1} and \cite[Lemma 3.5]{KMT2020_2}. It can be seen that for $0 < 2r \leq \rho \leq r_0$
    \begin{equation*}
    	A(r) \leq C\left(\frac{r}{\rho}\right) A(\rho)  + C\left(\frac{\rho}{r}\right)^2 \frac{1}{\rho^2}\int_{Q_{\rho}(z_0)} |v-v_\rho|^3 \,dz  =: C\left(\frac{r}{\rho}\right) A(\rho)  + C\left(\frac{\rho}{r}\right)^2 \bar{A}(\rho).
    \end{equation*}
	\textbf{Case (i).} Similarly to \cite[Lemma A.2]{WZ2014}, by using H\"{o}lder and Poincar\'{e} inequalities we obtain 
	\begin{equation*} 
		\bar{A}(\rho) \leq CE(\rho)^\frac{3(\alpha+\beta)}{2} \left(\rho^{1-\frac{2}{p_0}-\frac{3}{q_0}}  \|v-v_\rho\|_{L^{p_0}_tL^{q_0}_x (Q_\rho(z_0))}\right)^{3(1-\alpha-\beta)} \leq CE(\rho)^\frac{3(\alpha+\beta)}{2} M^{3(1-\alpha-\beta)},
	\end{equation*}
	where $0 < \rho \leq r_0$, $\alpha$ and $\beta$ should satisfy $\alpha \geq 0, \beta \geq 0$ and
	\begin{align*}
		1 &= \frac{3\alpha}{2} + \frac{\beta}{2} + \frac{3(1-\alpha-\beta)}{q_0},
		\quad 1 = \frac{3\beta}{2} + \frac{3(1-\alpha-\beta)}{p_0}, \quad 1 > \alpha + \beta, \quad \frac{2}{3} > \alpha, \beta,
		\\
		\text{i.e., } \alpha &= \frac{2}{3} \left(\frac{3}{p_0} + \frac{3}{q_0} - 2\right) \left(\frac{4}{p_0} + \frac{6}{q_0} - 3\right)^{-1}, \quad \beta =  \left(\frac{2}{p_0} + \frac{4}{q_0} - 2\right) \left(\frac{4}{p_0} + \frac{6}{q_0} - 3\right)^{-1}.
	\end{align*}
	In order to ensure $\alpha,\beta \geq 0$, we need to restrict ourselves to the cases
	\begin{equation*}
		\frac{3}{2} < \frac{2}{p_0} + \frac{3}{q_0} < 2,\quad \frac{3}{p_0} + \frac{3}{q_0} \geq 2,\quad \frac{2}{p_0} + \frac{4}{q_0} \geq 2, \quad p_0,q_0 < \infty.
	\end{equation*}
	However, other cases follow by using H\"{o}lder inequality. Indeed, for $(p_1,q_1)$ such that $1 \leq \frac{2}{p_1} + \frac{3}{q_1} < 2$ with $\frac{3}{2} < q_1 \leq \infty$ we can find $(p_0,q_0)$ satisfying the above conditions and 
	\begin{equation*}
		\rho^{1-\frac{2}{p_1}-\frac{3}{q_1}}  \|v-v_\rho\|_{L^{p_1}_tL^{q_1}_x (Q_\rho(z_0))} \leq C \rho^{1-\frac{2}{p_0}-\frac{3}{q_0}} \|v-v_\rho\|_{L^{p_0}_tL^{q_0}_x (Q_\rho(z_0))}.
	\end{equation*}
	In addition, we have $0 \leq 3(\alpha + \beta) < 2$ since our restriction on $(p_0,q_0)$.  By using the above estimates and Young inequality we find that for $0 < 3(\alpha + \beta) < 2$,
	\begin{align*}
		A(r) &\leq C\left(\frac{r}{\rho}\right) A(\rho)  + C\left(\frac{\rho}{r}\right)^2 \bar{A}(\rho) \qquad (0 < 2r \leq \rho \leq r_0)
		\\
		&\leq C\left(\frac{r}{\rho}\right) A(\rho) + C\left(\frac{\rho}{r}\right)^2 E(\rho)^\frac{3(\alpha+\beta)}{2} M^{3(1-\alpha-\beta)}
		\\
		&\leq C\left(\frac{r}{3\rho}\right) A(3\rho) + C\left(\frac{3\rho}{r}\right)^2 (C(1 + V_0 +  A(3\rho)))^\frac{3(\alpha+\beta)}{2} M^{3(1-\alpha-\beta)} \quad (0 < 6r \leq 3\rho \leq r_0)
		\\
		&\leq \left(\frac{C_0\theta}{3} + \frac{1}{4}\right) A(3\rho) + \theta^{-\frac{4}{2-3(\alpha+\beta)}} C_{\alpha,\beta,M,V_0}  \qquad (r = \theta\rho, 2\theta \leq 1) 
		\\
		&\leq \frac{1}{2} A(3\rho) + \theta^{-\frac{4}{2-3(\alpha+\beta)}} C_{\alpha,\beta,M,V_0}  \qquad (4C_0\theta \leq 3, 2\theta \leq 1, 3\rho \leq r_0), 
	\end{align*}
	where $C_0 > 1$ is an absolute constant, and for $\alpha + \beta = 0$ (i.e., $\alpha = \beta = 0$ or $p_0 = q_0 = 3$)
	\begin{equation*}
		A(\theta\rho) \leq \frac{1}{2} A(3\rho) + \theta^{-2} C_M, \qquad (4C_0\theta \leq 3, 2\theta \leq 1, 3\rho \leq r_0).
	\end{equation*}
	We choose $\theta = \min\{\frac{3}{4C_0},\frac{1}{2}\}$ and denote $\theta_0 = \frac{\theta}{3}$, $\rho_0 = 3\rho$. By iteration, 
	\begin{equation*}
		A(\theta^k_0\rho_0) \leq \frac{1}{2^k} A(\rho_0) + C_{\alpha,\beta,M,V_0}, \qquad  \forall k \geq 1, 0 < \rho_0 \leq r_0.
	\end{equation*}
	Thus, there exists $k_0 = k_0(r_0)$ sufficiently large such that
	\begin{equation*}
		A(r) \leq  C_{\alpha,\beta,M,V_0,r_0}, \qquad \forall 0 < r \leq \theta_0^{k_0}r_0.
	\end{equation*}
	\textbf{Case (ii).} Similarly to the previous case, by using H\"{o}lder, Poincar\'{e} and Sobolev inequalities we find that 
	\begin{equation*} 
		\bar{A}(\rho) \leq CE(\rho)^\frac{3\gamma_0}{2} \left(\rho^{2-\frac{2}{p_0}-\frac{3}{q_0}}  \|\nabla v\|_{L^{p_0}_tL^{q_0}_x (Q_\rho(z_0))}\right)^{3(1-\gamma_0)} \leq CE(\rho)^\frac{3\gamma_0}{2} M^{3(1-\gamma_0)},
	\end{equation*}
	where $0 < \rho \leq r_0$, $\gamma_0 = \alpha_0 + \beta_0$, $\alpha_0$ and $\beta_0$ should satisfy $\alpha_0 \geq 0, \beta_0 \geq 0$ and
	\begin{align*}
		1 &= \frac{3\alpha_0}{2} + \frac{\beta_0}{2} + \frac{3(1-\alpha_0-\beta_0)}{q^*_0}, \quad q^*_0 = \frac{3q_0}{3-q_0} \quad \text{for} \quad q_0 \in (1,3),
		\\
		\quad 1 &= \frac{3\beta_0}{2} + \frac{3(1-\alpha_0-\beta_0)}{p_0}, \quad 1 > \alpha_0 + \beta_0, \quad \frac{2}{3} > \alpha_0, \beta_0,
		\\
		\text{i.e., } \alpha_0 &= \frac{2}{3} \left(\frac{3}{p_0} + \frac{3}{q_0} - 3\right) \left(\frac{4}{p_0} + \frac{6}{q_0} - 5\right)^{-1}, \quad \beta_0 =  \left(\frac{2}{p_0} + \frac{4}{q_0} - \frac{10}{3}\right) \left(\frac{4}{p_0} + \frac{6}{q_0} - 5\right)^{-1}.
	\end{align*}	
	So that we force ourselves to the cases
	\begin{equation*}
		\frac{5}{2} < \frac{2}{p_0} + \frac{3}{q_0} < 3,\quad \frac{3}{p_0} + \frac{3}{q_0} \geq 3,\quad \frac{2}{p_0} + \frac{4}{q_0} \geq \frac{10}{3}, \quad q_0 \in (1,3)
	\end{equation*}
	and other cases can be recovered by using H\"{o}ler inequality as above. Note that we have $0 \leq 3\gamma_0 < 2$. Therefore, we can bound $A(r)$ as in  Case (i). 
	\\
	\textbf{Case (iii).} Let $w = \nabla \times v$ and $0 < \rho \leq r_0$, as in \cite[Lemma 3.6]{GKT2007}, we define 
	\begin{equation*}
		u(x,t) := \text{p.v.} \frac{1}{4\pi} \int_{\mathbb{R}^3} \nabla_x\left(\frac{1}{|x-y|}\right) \times w(y,t) \phi(y) \,dy \quad \text{and} \quad h := v - u,
	\end{equation*}
	where $\phi$ is a standard cut off function supported in $B_\rho(x_0)$ such that $\phi = 1$ in $B_{\frac{3\rho}{4}}(x_0)$. It can be seen that $\Delta h = 0$ in $B_{\frac{3\rho}{4}}(x_0)$. For $0 < 2r \leq \rho \leq r_0$ and $q_0 \in (1,\infty)$ we find that 
	\begin{align*}
		\|\nabla v\|^{q_0}_{L^{q_0}(B_r(x_0))} &\leq C_{q_0} \|\nabla u\|^{q_0}_{L^{q_0}(B_r(x_0))} + C_{q_0} \|\nabla h\|^{q_0}_{L^{q_0}(B_r(x_0))}
		\\
		&\leq C_{q_0} \|w\|^{q_0}_{L^{q_0}(B_{\rho}(x_0))} + C_{q_0} \left(\frac{r}{\rho}\right)^3 \|\nabla h\|^{q_0}_{L^{q_0}(B_{\rho}(x_0))}
		\\
		&\leq C_{q_0} \|w\|^{q_0}_{L^{q_0}(B_{\rho}(x_0))} + C_{q_0} \left(\frac{r}{\rho}\right)^3 \|\nabla v\|^{q_0}_{L^{q_0}(B_{\rho}(x_0))},
	\end{align*}
	where we used the Calder\'{o}n-Zygmund estimate and the mean value property of harmonic functions. Integrating in time yields for $ 1 < q_0 < 3$ (other cases can be recovered by using H\"{o}lder inequality)
	\begin{align*}
		H(r) &:= r^{2-\frac{2}{p_0}-\frac{3}{q_0}} \|\nabla v\|_{L^{p_0}_tL^{q_0}_x(Q_r(z_0))}  \qquad (0 < r \leq r_0)
		\\
		&\leq C_{p_0,q_0} \left(\frac{r}{\rho}\right)^{2-\frac{2}{p_0}-\frac{3}{q_0}} M + C_{p_0,q_0} \left(\frac{r}{\rho}\right)^{2-\frac{2}{p_0}} H(\rho) \qquad (0 < 2r \leq \rho \leq r_0)
		\\
		&= C_{p_0,q_0} \theta^{2-\frac{2}{p_0}-\frac{3}{q_0}} M + C_{p_0,q_0} \theta^{2-\frac{2}{p_0}} H(\rho) \qquad (r = \theta \rho, 2\theta \leq 1, 0 < \rho \leq r_0)
		\\
		&\leq C_{p_0,q_0} \theta^{2-\frac{2}{p_0}-\frac{3}{q_0}} M + C_{p_0,q_0} \theta^{\frac{3}{q_0}-1} H(\rho) \qquad (r = \theta \rho, 2\theta \leq 1, C_{p_0,q_0} > 1, 0 < \rho \leq r_0)
		\\
		&\leq C_{p_0,q_0} M + \frac{1}{2} H(\rho) \qquad (2C_{p_0,q_0} \theta^{\frac{3}{q_0}-1} \leq 1, 2\theta \leq 1, C_{p_0,q_0} > 1, 0 < \rho \leq r_0).
	\end{align*}
	By iteration, we can bound $H(r)$ by $C_{p_0,q_0} M$ for $0 < r \leq \theta^{k_0}r_0$ with some sufficiently large integer $k_0(r_0)$. It returns to Case (ii).
	\\
	\textbf{Case (iv).} For $0 < 2r \leq \rho \leq r_0$, $1 \leq q_0 < \frac{3}{2}$ (other cases where $\frac{3}{2}\leq q_0 \leq \infty$ can be recovered by using H\"{o}lder inequality) and $\frac{3}{2} \leq \bar{q}_0 := \frac{3q_0}{3-q_0} < 3$, using Sobolev inequality yields
	\begin{align*}
	\|w\|^{\bar{q}_0}_{L^{\bar{q}_0}(B_r(x_0))} &\leq C_{q_0} \left(\frac{r}{\rho}\right)^3 \|w\|^{\bar{q}_0}_{L^{\bar{q}_0}(B_{\rho}(x_0))} + C_{q_0} \|w-w_\rho\|^{\bar{q}_0}_{L^{\bar{q}_0}(B_{\rho}(x_0))}
	\\
	&\leq C_{q_0} \left(\frac{r}{\rho}\right)^3 \|w\|^{\bar{q}_0}_{L^{\bar{q}_0}(B_{\rho}(x_0))} + C_{q_0} \|\nabla w\|^{\bar{q}_0}_{L^{q_0}(B_{\rho}(x_0))},
	\end{align*}
	which implies that 
	\begin{align*}
		G(r) &:= r^{2-\frac{2}{p_0}-\frac{3}{\bar{q}_0}} \| w\|_{L^{p_0}_tL^{\bar{q}_0}_x(Q_r(z_0))}  \qquad (0 < r \leq r_0)
		\\
		&\leq C_{p_0,q_0} \theta^{\frac{3}{q_0}-2} G(\rho) + C_{p_0,q_0} \theta^{2-\frac{2}{p_0}-\frac{3}{q_0}} M \qquad (0 < \rho \leq r_0, r = \theta \rho, 2\theta \leq 1,C_{p_0,q_0} > 1)
		\\
		&\leq \frac{1}{2} G(\rho) + C_{p_0,q_0} M \qquad (2C_{p_0,q_0} \theta^{\frac{3}{q_0}-2} \leq 1, 2\theta \leq 1, 0 < \rho \leq r_0).
	\end{align*}
	Therefore, $G(r)$ is bounded by $C_{p_0,q_0} M$ for $0 < r \leq \theta^{k_0}r_0$ with some sufficiently large integer $k_0(r_0)$. We then return to Case (iii) since $\frac{2}{p_0} + \frac{3}{\bar{q}_0} = \frac{2}{p_0} + \frac{3}{q_0} - 1 \in [2,3)$. All above cases imply that 
	\begin{equation*}
		E(r) \leq  C_{\alpha,\alpha_0,\beta,\beta_0,M,V_0,p_0,q_0,r_0}, \qquad \forall 0 < r \leq r_0.
	\end{equation*}
	The rest of the proof is similar to that of Theorem \ref{theo1} by using \eqref{v3z0} instead of \eqref{v3v0}, where we only need to modify the definition of $(v_\delta,\pi_\delta)$ in the following way
	\begin{equation*} 
		v_\delta(x,t) := \delta v(\delta x + x_0,\delta^2 t + t_0)
		\quad 
		\text{and}
		\quad
		\pi_\delta(x,t) := \delta^2 \pi_{r_0,x_0}(\delta x + x_0,\delta^2 t + t_0).
	\end{equation*}
	
	$\bullet$\,\textnormal{(}The case of suitable weak solution\textnormal{)}
	The local energy inequality yields for $0 < 2r \leq r_0$
	\begin{align*}
		E(r) &\leq C\left(1 +  \frac{1}{4r^2} \|v\|^3_{L^3(Q_{2r}(z_0))}  + \frac{1}{4r^2} \|\pi|^\frac{3}{2}_{L^\frac{3}{2}(Q_{2r}(z_0))}\right) =: C(1 + A(2r) + B(2r)),
	\end{align*}
	where $E(r)$ is given as in the case of local energy solutions with $\pi$ instead of $\pi_{r_0,x_0}$. As in \cite[Lemma 3.4]{GKT2007} we have
	\begin{equation*}
		B(r) \leq C\left(\frac{r}{\rho}\right) B(\rho) + C \left(\frac{\rho}{r}\right)^2 \bar{A}(\rho) \qquad \text{for} \quad 0 < 2r \leq \rho \leq r_0,
	\end{equation*}
	which imples as in the previous case that
	\begin{align*}
		D(r) &:= A(r) + B(r) \qquad (0 < 2r \leq r_0)
		\\
		&\leq C\left(\frac{r}{\rho}\right) D(\rho)   + C\left(\frac{\rho}{r}\right)^2 \bar{A}(\rho) \qquad (0 < 8r \leq 4\rho \leq r_0)
		\\
		&\leq C\left(\frac{r}{\rho}\right) D(\rho)  + C\left(\frac{\rho}{r}\right)^2 E(\rho)^\frac{3(\alpha+\beta)}{2} M^{3(1-\alpha-\beta)}
		\\
		&\leq C\left(\frac{r}{2\rho}\right) D(2\rho) + C\left(\frac{\rho}{r}\right)^2 (C(1 + D(2\rho)))^\frac{3(\alpha+\beta)}{2} M^{3(1-\alpha-\beta)}. 
	\end{align*}
	The rest of the proof then can follows as in the case of local energy solutions by considering $D(r)$ as $A(r)$. We omit the details and end the proof.
	\end{proof}
	
	%
	\section{Proof of Theorem \ref{theo3}}\label{sec:theo3}
	%
	
	We will present the proof of Theorem \ref{theo3}.
	The idea of the proof comes from those of \cite[Theorems 2.1]{KWM2017,KWM2019}. However, for local energy solutions the condition is not involved the pressure due to the decomposition \eqref{pressure_decomposition}.
	
	\begin{proof}[Proof of Theorem \ref{theo3}]\quad $\bullet$\,\textnormal{(}The case of local energy solution\textnormal{)}\, We assume that the conclusion is not true, i.e., there exists a sequence of local energy solutions $(v^n,\pi^n)$ such that 
	\begin{equation*}
	\left\{
	\begin{aligned}
		r_0^{-2} \|v^n\|^3_{L^3(Q_{r_0}(z_0))}  &\leq M,
		\\
		r^{1-\frac{2}{p}-\frac{3}{q}}_0 \|v^n_3\|_{L^p_tL^q_x(Q_{r_0}(z_0))} &\to 0 \qquad \text{as}\quad n \to \infty \quad \text{for} \quad 1 \leq p,q \leq \infty,
		\\
		(\delta_0 r_0)^{-2} \|v^n\|^3_{L^3(Q_{\delta_0 r_0}(z_0))} &> \epsilon_0 \qquad \forall \delta_0 \in (0,1).
	\end{aligned}
	\right.
	\end{equation*}	
	As in the proof of Theorem \ref{theo2},
	\begin{equation*}
		E^n(r) \leq C(1 + M + V_0) \qquad \text{for} \quad r := \frac{r_0}{3},
	\end{equation*}
	where $E^n(r)$ is given as in the proof of Theorem \ref{theo2} with $(v^n,\pi^n_{r_0,x_0})$ instead of $(v,\pi_{r_0,x_0})$. Therefore, there exist a subsequence of $(v^n,\pi^n_{r_0,x_0})$ (still denoted by $(v^n,\pi^n_{r_0,x_0})$) and $(\bar{v},\bar{\pi})$ such that as $n \to \infty$
	\begin{equation*} 
		\left\{
		\begin{aligned}
			v^n  &\overset{\ast}{\rightharpoonup} \bar{v} \qquad \text{ in } L^\infty_tL^2_x(Q_{r}(z_0))),
			\\
			\nabla v^n  &\overset{}{\rightharpoonup} \nabla \bar{v} \quad\, \text{ in } L^2(Q_{r}(z_0)),
			\\
			\pi^n_{r_0,x_0}  &\overset{}{\rightharpoonup} \bar{\pi} \qquad \text{ in } L^\frac{3}{2}(Q_{r}(z_0)),
			\\
			v^n &\to \bar{v} \qquad \,\,\text{in } L^a(Q_{r}(z_0)) \qquad \text{for} \quad 1 \leq a < \frac{10}{3}.
		\end{aligned}
		\right.
	\end{equation*}
	In addition, we have $(\bar{v},\bar{\pi})$ is a suitable weak solution to \eqref{NS} in $Q_{r_0}(z_0)$. The convergence of $v^n_3$ implies that $\bar{v}_3 \equiv 0$. Therefore, $(\bar{v},\bar{\pi})$ is also a suitable weak solution to "the limiting system" and $\bar{v}$ is regular as well. Taking $n \to \infty$ yields
	\begin{equation*}
		(\delta_0 r)^{-2} \|v^n\|^3_{L^3(Q_{\delta_0 r}(z_0))} \to (\delta_0 r)^{-2} \|\bar{v}\|^3_{L^3(Q_{\delta_0 r}(z_0))} \geq \epsilon_0 \qquad \forall \delta_0 \in (0,1),
	\end{equation*}
	which is a contradiction with the regularity of $\bar{v}$ for $\delta_0$ small enough. Therefore, the proof is complete.

	$\bullet$\,\textnormal{(}The case of suitable weak solution\textnormal{)} The proof in this case is similar to that of for local energy solutions. We define $E^n(r)$ as in the proof of Theorem \ref{theo2} with $(v^n,\pi^n)$ instead of $(v,\pi_{r_0,x_0})$. Thus, 
	\begin{equation*}
		E^n(r) \leq C(1 + M) \qquad \text{for} \quad r := \frac{r_0}{2},
	\end{equation*}
	that completes the proof as in the previous case.
	\end{proof}
	
	%
	\section*{Acknowledgments} 
	%
	K. Kang was supported by NRF-2019R1A2C1084685. 
    D. D. Nguyen was supported by NRF-2015R1A5A1009350.


\end{document}